\documentclass[12pt,reqno]{amsart}
\usepackage{amssymb}
\usepackage{amsmath, mathtools}

\usepackage{amsthm}

\usepackage{amscd}

\newcommand{\RNum}[1]{\uppercase\expandafter{\romannumeral #1\relax}}

\usepackage{caption}

\usepackage[T2A]{fontenc}
\usepackage[utf8]{inputenc}
\usepackage[english]{babel}

\input{int.def} 

\usepackage[sort]{cite}
\usepackage{tikz-cd}
\usetikzlibrary{cd}
\usepackage{dirtytalk}
\usepackage[linktoc=page, colorlinks, linkcolor=blue, citecolor=blue]{hyperref}

\usepackage{xcolor}
\usepackage{centernot}

\usepackage{enumitem}

\usepackage{pgfplots}
\usepackage{multicol}

\pgfplotsset{compat=1.17}

\makeatletter
\renewenvironment{proof}[1][\proofname]{%
  \par\vspace{\topsep}%
  \normalfont\topsep6\p@\@plus6\p@\relax
  \trivlist
  \item[\hskip\labelsep\itshape #1\@addpunct{.}]\ignorespaces
}{%
  \endtrivlist
}
\makeatother

\numberwithin{equation}{section}

\DeclarePairedDelimiterX \ip[2]{\langle}{\rangle}{#1,#2}
\DeclarePairedDelimiterXPP \Prob[1]{\mathbb{P}}\{\}{}{ #1} 
\DeclarePairedDelimiterXPP \Probevent[1]{\mathbb{P}}(){}{#1} 

\def \R {\mathbb{R}}

\usepackage[mathcal]{euscript}

\usepackage{titlesec}
\titleformat{\section}[runin]{\bfseries}{\thesection.}{3pt}{}[.]

\usepackage{geometry}
\newgeometry{vmargin={25mm}, hmargin={22mm,22mm}, footskip=10mm}   

\begin{document}
\title[
First integrals of dense 
hard-ball gases
]{First integrals of dense 
hard-ball gases}

\author{Gleb Smirnov}
\address{
Mathematical Sciences Institute, 
Australian National University, Canberra, Australia}
\email{gleb.smirnov@anu.edu.au}


\begin{abstract}
We answer two questions of Kozlov about polynomial-in-momentum first integrals of billiard systems. For billiards in bounded polytopes, such integrals are independent of position. For hard-ball gases in a rectangular box, every polynomial-in-momentum first integral is a polynomial in the total kinetic energy whenever the collision graph is connected. 
\end{abstract}

\maketitle
\setcounter{section}{0}

\section{Main results}\label{intro} 
A Birkhoff billiard describes a point particle moving inside a domain
\(K\subset\R^d\). The particle moves with constant velocity and reflects
elastically from the boundary; thus, between collisions,
\[
\dot q=p,\qquad \dot p=0,
\]
while at a regular boundary point with unit normal \(n\),
\[
p\longmapsto p-2\langle p,n\rangle n.
\]
A first integral is a function constant along the billiard trajectories. We study \emph{polynomial-in-momentum} first integrals:
\[
f(q,p)=\sum_{|\alpha|\le r}a_\alpha(q)p^\alpha,
\]
with no regularity assumption on the coefficients \(a_\alpha\).  
\smallskip%

The basic example of such a 
first integral is the kinetic energy
\[
H(p)=\frac12\|p\|^2.
\]
Our first result concerns billiards in polytopes.
\begin{theorem}\label{thm:birkhoff}
Let \(d\ge2\). Let \(K\subset\R^d\) be a bounded, not necessarily convex,
polytope with nonempty interior. Every polynomial-in-momentum first integral
of the Birkhoff billiard in \(K\) is independent of the position:
\[
f(q,p)=Q(p).
\]
Consequently, if \(G_K\subset O(d)\) is the closed group generated by the
linear reflections in the facets of \(K\), then \(Q\) is \(G_K\)-invariant.
In particular, if \(G_K=O(d)\), then \(f\) is a polynomial in the kinetic
energy.
\end{theorem}

Let us make the class of 
polytopes explicit. If \(K\) is convex, 
then:
\[
K=\{x\in\R^d:\ell_j(x)\le0,\ 1\le j\le m\},
\]
for irredundant affine linear functions \(\ell_1,\ldots,\ell_m\), and the
facets are
\[
F_j=K\cap\{\ell_j=0\}.
\]
More generally, by a non-convex polytope we mean a finite union of convex polytopes whose interior is connected.
\smallskip%

Our second result concerns the 
Boltzmann-Gibbs gas. Consider \(N\) identical
balls of radius \(\rho>0\) and unit mass moving in the rectangular box
\[
B=\prod_{k=1}^d[0,L_k],\qquad L_k>2\rho.
\]
The balls move freely and collide elastically with one another and with the
walls. Their centers lie in
\[
D=\prod_{k=1}^d(\rho,L_k-\rho),
\]
so the collision-free configuration space is
\[
\Omega=
\left\{
(q_1,\ldots,q_N)\in D^N:
\|q_i-q_j\|>2\rho\ \text{for }i\ne j
\right\}.
\]
The space \(\Omega\) need not 
be connected. So fix a connected component
\(\Omega_\gamma\). Its \emph{collision graph} \(\Gamma_\gamma\) has vertices
\(1,\ldots,N\). We join \(i\) and \(j\) if the balls can be continuously moved, while remaining in \(\Omega_\gamma\), to a configuration in which balls \(i\) and \(j\)
collide and no other collision occurs.
\smallskip%

Write:
\[
p=(p_1,\ldots,p_N),\qquad p_i\in\R^d,
\]
for the momenta of the balls. A polynomial-in-momentum first integral has the form:
\[
f(q,p)=\sum_{|\alpha|\le r}a_\alpha(q)p^\alpha.
\]
The basic example is the total
kinetic energy
\[
H(p)=\frac12\sum_{i=1}^N\|p_i\|^2.
\]
\begin{theorem}\label{thm:gas}
Let \(N\ge2\) and \(d\ge2\). If the collision graph
\(\Gamma_\gamma\) is connected, then every polynomial-in-momentum first
integral of the Boltzmann-Gibbs gas on \(\Omega_\gamma\) is a polynomial
in the total kinetic energy.
\end{theorem}

The connectedness assumption is 
necessary. If \(C_1,C_2,\ldots\) are the connected components of
\(\Gamma_\gamma\), then each
\[
H_{C_s}(p)
=
\frac12\sum_{i\in C_s}\|p_i\|^2
\]
is also a first integral.
\smallskip%

We now discuss the motivation and related work.
\smallskip%

The physical problem behind the 
Boltzmann-Gibbs gas is considerably deeper
than Theorem~\ref{thm:gas}. One would like to understand the ergodic and
statistical behavior of a gas of elastically colliding particles. The
Boltzmann-Gibbs gas (aka the hard-ball system) is itself a billiard:
the positions of all \(N\) balls form one point in the \(dN\)-dimensional
configuration space, and collisions become reflections at its boundary.
\smallskip%

The ergodic theory of billiards goes back to the foundational work of Sinai
\cite{Sinai}. Much of the subsequent theory of hard-ball systems was
developed on a flat torus, that is, with periodic boundary conditions; see,
for example, \cite{Sinai-Chernov,Simanyi-2,Simanyi-3, Simanyi-4} and the references
therein. In a rectangular box the walls also destroy conservation of total
momentum. For this model, Sim\'anyi proved ergodicity for two balls
\cite{Simanyi-1}. For 
surveys of the statistical and physical aspects of hard-ball systems, see
\cite{Szasz,ChernovYoung,Liverani}.
\smallskip%

An additional first integral is an obstruction to ergodicity: if it is
functionally independent of the energy, its level sets give additional
invariant subsets of an energy 
surface. Absence of additional first integrals
is much weaker than ergodicity, but is more accessible algebraically. Kozlov \cite{Kozlov} developed this viewpoint for the Boltzmann-Gibbs gas. He pointed out a possible thermodynamic application: if sufficiently
regular stationary densities depend only on conserved quantities, then the
absence of additional conservation laws forces them to depend only on the
energy. This suggests a route from non-integrability to equilibrium
statistical mechanics. We do not pursue this question here, but it provides an additional motivation for studying the absence of first integrals.
\smallskip%

We restrict attention to integrals polynomial in the momenta. Many classical
integrable systems have additional integrals of this form. The basic billiard
example is the ellipse, which has an additional quadratic integral. Its
integrability is closely related to Poncelet's porism and the
Arnold-Liouville theorem; see, e.g.,
\cite{Levi-Tabachnikov,Izo}. The problem of characterizing integrable convex billiards goes back to
Birkhoff and remains active; see, e.g.,
\cite{Mather,Treschev,Avila-DSim-Kalosh,Kalosh-Sorr,Bialy-Mir},
and \cite{Tabachnikov} for an introduction.
\smallskip%

Our first theorem concerns a nonsmooth class of billiards. For a rectangular
box, polynomial first integrals may depend on
the momenta, but not on the position. In
\cite[\S9, item \(6^\circ\)]{Kozlov}, Kozlov posed the problem of extending
this position-independence statement from rectangular boxes to arbitrary
polytopes. He observed that solving this problem would extend his
multidimensional Lorentz-gas result from rectangular boxes to polyhedral
containers. Theorem~\ref{thm:birkhoff} solves this problem.
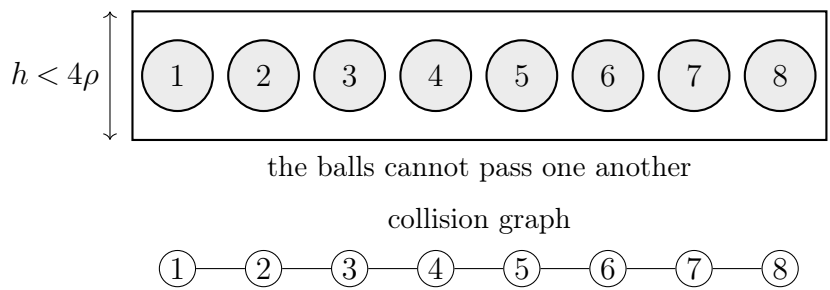
\begin{figure}[ht]
\centering
\begin{tikzpicture}[scale=0.85]


\draw[thick] (0,0) rectangle (10.8,2);

\foreach \x/\lab in {
0.7/1,
2.04/2,
3.38/3,
4.72/4,
6.06/5,
7.40/6,
8.74/7,
10.08/8}
{
  \draw[thick,fill=gray!15] (\x,1) circle (0.55);
  \node at (\x,1) {\(\lab\)};
}

\draw[<->] (-0.35,0) -- (-0.35,2);
\node[left] at (-0.35,1) {\(h<4\rho\)};

\node at (5.4,-0.45)
{\small the balls cannot pass one another};

\begin{scope}[yshift=-2cm]

\node at (5.4,0.75) {\small collision graph};

\foreach \x/\lab in {
0.7/1,
2.04/2,
3.38/3,
4.72/4,
6.06/5,
7.40/6,
8.74/7,
10.08/8}
{
  \node[circle,draw,inner sep=1.4pt] (v\lab) at (\x,0) {\(\lab\)};
}

\draw
(v1)--(v2)--(v3)--(v4)--(v5)--(v6)--(v7)--(v8);

\end{scope}

\end{tikzpicture}

\caption*{
A dense hard-ball gas in a narrow channel, with \(2\rho<h<4\rho\).
The balls cannot pass one another, so the collision-free configuration
space is disconnected. Nevertheless, within each ordering component,
neighboring balls can collide and the collision graph is connected.
}
\label{fig:dense-gas}
\end{figure}

Our second theorem is more directly motivated by the dynamics of a gas. Kozlov proved in Theorem~11 of \cite{Kozlov} that for a \emph{thin} Boltzmann-Gibbs gas, when the whole configuration space \(\Omega\) is connected, every polynomial-in-momentum first integral is a polynomial in the total kinetic
energy. In \cite[\S9, item \(1^\circ\)]{Kozlov}, he asked whether the same
conclusion remains true for dense 
gases. Theorem~\ref{thm:gas} confirms this under the weaker assumption: the collision graph is connected. This substantially extends the range of physical configurations covered by
the result. For example, the theorem applies to balls in a narrow channel which cannot pass one another. 

\section{Polynomial continuation lemma}
To begin with, consider the 
free-motion part of the billiard dynamics:
\[
\dot q=p,\qquad \dot p=0.
\]
Starting from \((q,p)\), the trajectory is 
\((q(t),p(t))=(q+tp,p)\). 
For a billiard, this is valid until the trajectory reaches the boundary.
Thus, if \(U\subset\R^d\) is the region of free motion, every first
integral satisfies:
\begin{equation}\label{eq:free-flight}
f(q+tp,p)=f(q,p)
\end{equation}
whenever the segment from \(q\) to \(q+tp\) is contained in \(U\). We call \eqref{eq:free-flight} the \emph{free-flight equation}.
\smallskip%

The following continuation lemma is due to Kozlov \cite{Kozlov} under
the assumption that the coefficients of the polynomial in
the momenta are \(C^1\). We shall not assume any regularity in \(q\).

\begin{lemma}[Polynomial continuation]\label{lem:kozlov}
Let \(U\subset\R^d\) be an 
open subset. Let
\(f \colon U\times\R^d\to\R\)
be polynomial in \(p\), with arbitrary dependence on \(q\). Assume
\[
f(q+tp,p)=f(q,p)
\]
whenever the segment from \(q\) to \(q+tp\) is contained in \(U\). Then, on every connected component of \(U\), \(f\) is a polynomial in \((q,p)\).
\end{lemma}
\begin{proof}
Let \(B\Subset U\) be an open 
ball. For \(x,y\in B\), set:
\[
g(x,y)=f(x,y-x).
\]
Since the segment from \(x\) to \(y\) lies in \(B\), the free-flight
equation gives:
\[
g(x,y)=f(y,y-x).
\]
Let \(r=\mathrm{deg}_p f\). For every fixed \(x\), the function 
\(y \to g(x,y)=f(x,y-x)\) 
is a polynomial of degree at most \(r\). On the other hand, using the
identity above, for every 
fixed \(y\), \(x\to g(x,y)=f(y,y-x)\)
is also a polynomial of degree at most \(r\). It follows that \(g\) is jointly polynomial in \((x,y)\). 
The proof now follows by polynomial interpolation. \qed
\end{proof}

\section{Proof of Theorem \ref{thm:birkhoff}} 

Let \(K\subset\R^d\) be a connected polytope, and let \(f\) be a
polynomial-in-momentum first 
integral of the Birkhoff billiard in 
\(K\). By 
Lemma~\ref{lem:kozlov}, \(f\) admits a unique polynomial
extension to the whole space \(\R^{2d}\). Moreover, this extension satisfies the free-flight equation globally:
\[
f(q+tp,p)=f(q,p),
\qquad \forall q, p\in\R^d,\quad t\in\R.
\]
Abusing notation, we use the same letter \(f\) for the original first
integral and for its polynomial  continuation.
\smallskip%

Let \(F\) be a facet of \(K\), and let \(H\) be the affine hyperplane
containing \(F\). Denote by
\[
\sigma_H:\R^d\to\R^d
\]
the affine reflection in \(H\), and by
\[
R_H\in O(d)
\]
its linear part.
\begin{lemma}\label{lem:reflection}
The polynomial \(f\) satisfies:
\[
f(\sigma_H q, R_H p) = f(q,p),
\qquad q,p\in\R^d.
\]
\end{lemma}
\begin{proof}
Let \(F\) be a facet of \(K\), contained in \(H\). Let \(x\) lie in the interior of \(F\), and let 
\(p\) be transverse to \(H\). 
The reflection law gives:
\[
f(x,p)=f(x,R_Hp).
\]
Since \(f\) is polynomial, it follows that:
\[
f(x,p)=f(x,R_Hp),
\qquad \forall x\in H,\quad p\in\R^d.
\]
Suppose again that \(p\) is not parallel to \(H\). For any \(q\), let \(t\) be defined by
\[
q + t p \in H.
\]
Then:
\[
q + tp - t R_H p = \sigma_H q.
\]
Hence:
\[
\begin{aligned}
f(q,p)
&=f(q+tp,p)
&&\text{(free-flight equation)}\\
&=f(q+t p, R_H p)
&&\text{(wall reflection)}\\
&=f(q + tp - t R_H p, R_Hp)
&&\text{(free-flight equation),}
\end{aligned}
\tag{2}\label{eq}
\]
and the proof follows. \qed
\end{proof}

\begin{proof}[Proof of 
Theorem \ref{thm:birkhoff}] 
To begin with, reduce to the case when \(K\) is convex. Let
\(H_1,\ldots,H_m\) be the hyperplanes containing the facets of \(K\). Their arrangement decomposes \(K\) into finitely many convex 
polytopes \(K_1,\ldots,K_s\). Every facet of each \(K_j\) is contained in one of the \(H_i\). By the
previous lemma, \(f\) has the corresponding global reflection symmetry; hence, its restriction to \(K_j\) is a first integral of the billiard in \(K_j\). It is therefore enough to consider convex \(K\).
\smallskip%

Let \(H_1,\ldots,H_m\) be the affine hyperplanes containing the facets
of \(K\), and let \(\sigma_j\) and \(R_j\) be the corresponding affine
and linear reflections. By the previous lemma,
\begin{equation}\label{eq:global-reflection}
f(\sigma_j q,R_jp)=f(q,p),
\qquad j=1,\ldots,m.
\end{equation}
Choose unit normals \(n_j\) and numbers \(c_j\) so that:
\[
K=\{x:\ell_j(x)\le0,\ j=1,\ldots,m\},
\qquad
\ell_j(x)=\langle n_j,x\rangle-c_j,
\]
and
\[
H_j=\{x:\ell_j(x)=0\}.
\]
We use the reflection walk of Motzkin and Schoenberg
\cite{Motzkin-Sch}. Take any \((q,p)\), and set:
\[
q_0=q,\quad p_0=p.
\]
If \(q_k\notin K\), choose any \(j\) such that \(\ell_j(q_k)>0\) and set:
\[
q_{k+1}=\sigma_jq_k,
\quad
p_{k+1}=R_jp_k.
\]
By \eqref{eq:global-reflection}, for each \(k\),
\[
f(q_k,p_k)=f(q,p).
\]
Choose \(x_*\in\operatorname{int}K\), and set:
\[
\delta_j=-\ell_j(x_*)>0.
\]
Observe:
\[
\|q_{k+1}-x_*\|^2 \le \|q_k-x_*\|^2.
\]
Hence, \((q_k)\) is 
bounded. Consequently, there is a subsequence of \(q_k\) 
converging to some \(q_\infty\). In fact, the whole sequence converges to \(q_\infty\). Indeed, our calculation gives:
\[
\|q_{k+1}-x_*\| \le 
\|q_{k}-x_*\|\quad \text{for any 
\(x_* \in K\)}.
\]
It follows that \(\|q_k - x_{*}\|\) 
has a limit for every \(x_{*}\in K\). 
Any two accumulation points of \(q_k\) have the same distance from every
\(x^{*} \in K\). If \(K\) has non-empty interior, they must coincide.
\smallskip%

We claim that 
\(q_\infty\in K\). Choose a 
small enough neighborhood \(V\) of \(q_\infty\) such that:
\[
\sigma_j(V)\cap V=\varnothing
\qquad\text{whenever }q_\infty\notin H_j.
\]
For all large \(k\), both \(q_k\) and \(q_{k+1}\) lie in
\(V\); hence, every wall used from that moment on contains 
\(q_\infty\). Such reflections fix \(q_\infty\); therefore:
\[
\|q_{k+1}-q_\infty\|
=
\|q_k-q_\infty\|.
\]
Since \(q_k\to q_\infty\), 
it follows that \(q_k=q_\infty\) for all large \(k\). If \(q_\infty\notin K\), some inequality is strictly violated at \(q_\infty\), so the 
next reflection
does not fix \(q_\infty\). This is a contradiction. Thus: 
\[
q_{\infty} \in K.
\]
Finally, \(\|p_k\|=\|p\|\). Passing to a subsequence, assume:
\[
p_k\to p_\infty,
\quad
\|p_\infty\|=\|p\|.
\]
Since \(f(q_k,p_k)=f(q,p)\),
\[
f(q,p)=f(q_\infty,p_\infty).
\]
Thus:
\[
|f(q,p)|
\le
\sup_{\substack{x\in K\\ \|v\|=\|p\|}}
|f(x,v)|.
\]
Consequently, the
polynomial \(q\to f(q,p)\)
is bounded on \(\R^d\), and therefore constant. This completes the 
proof. \qed
\end{proof}

\section{Proof of Theorem~\ref{thm:gas}}
To begin with, view the Boltzmann--Gibbs gas as a billiard in its
\(dN\)-dimensional configuration space \(\Omega_\gamma\). Write
\[
q=(q_1,\ldots,q_N)
\]
for the centers of the balls, and
\[
p=(p_1,\ldots,p_N)
\]
for their momenta. Between collisions,
\[
\dot q=p,\qquad \dot p=0.
\]
The boundary of \(\Omega_\gamma\) has two kinds of smooth pieces. A wall
collision of the \(i\)-th ball corresponds to
\[
q_{i,k}=\rho
\qquad\text{or}\qquad
q_{i,k}=L_k-\rho.
\]
If \(n\) is the unit normal to the wall, 
the collision changes the
momenta as follows:
\[
p_i'=p_i-2\langle p_i,n\rangle n,
\]
with all other momenta unchanged. This is precisely reflection of the
\(i\)-th ball from the wall.
\smallskip%

A collision of balls \(i\) and \(j\) corresponds to
\[
\|q_i-q_j\|=2\rho.
\]
Set:
\[
n=\frac{q_i-q_j}{\|q_i-q_j\|}.
\]
A unit normal to this hypersurface in \(\R^{dN}\) is
\[
N_{ij}(n)
=
\frac1{\sqrt2}
(0,\ldots,n,\ldots,-n,\ldots,0),
\]
where \(n\) occupies the \(i\)-th position and \(-n\) 
the \(j\)-th. The collision changes the
momenta as follows:
\[
\begin{aligned}
p_i'&=p_i-\langle p_i-p_j,n\rangle n,\\
p_j'&=p_j+\langle p_i-p_j,n\rangle n,
\end{aligned}
\]
with all other momenta unchanged. This is precisely the elastic collision
law for two equal masses.
\smallskip%

Thus the gas on \(\Omega_\gamma\) is 
a billiard. Since
\(\Omega_\gamma\) is connected, Lemma~\ref{lem:kozlov} applies. 
It follows that every polynomial-in-momentum first integral extends uniquely to a polynomial \(f(q,p)\) on the whole space \(\R^{dN}_q\times\R^{dN}_p\). In particular, the free-flight
equation holds globally:
\[
f(q+tp,p)=f(q,p),
\quad q,p\in\R^{dN},\ t\in\R.
\]
Let \(H\subset\R^d\) be one of the 
hyperplanes
\[
x_k=\rho
\qquad\text{or}\qquad
x_k=L_k-\rho.
\]
Denote by \(\sigma_H:\R^d\to\R^d\) the affine reflection in \(H\), and
by \(R_H\in O(d)\) its linear part. For the \(i\)-th ball, define:
\[
\sigma_H^{(i)}(q_1,\ldots,q_N)
=
(q_1,\ldots,\sigma_H q_i,\ldots,q_N),
\]
and
\[
R_H^{(i)}(p_1,\ldots,p_N)
=
(p_1,\ldots, R_H p_i,\ldots,p_N).
\]
\begin{lemma}[Wall symmetry]\label{lem:ball-wall}
Suppose that the \(i\)-th ball can collide with the wall \(H\). Then:
\[
f(\sigma_H^{(i)}q,R_H^{(i)}p)=f(q,p),
\qquad q,p\in\R^{dN}.
\]
\end{lemma}
\begin{proof}
Let \(\overline{\Omega}_\gamma\) denote 
the closure of \(\Omega_\gamma\) in \(\R^{N d}\). Set:
\[
\mathcal H
=
\{q\in\R^{dN}:q_i\in H\}
\]
Since the \(i\)-th 
ball can collide with \(H\), 
\[
\mathcal H \cap \overline{\Omega}_\gamma
\]
contains a nonempty relatively open subset of \(\mathcal H\). At every point \(x\) of this subset, the reflection law gives:
\[
f(x,p)=f(x,R_H^{(i)}p)\quad 
\text{for every \(p\) 
transverse to \(\mathcal H\).}
\]
Since \(f\) is polynomial, it 
follows that:
\[
f(x,p)=f(x,R_H^{(i)}p),
\quad
x\in\mathcal H,\quad p\in\R^{dN}.
\]
Suppose that \(p_i\) is not parallel to \(H\). For any \(q\in\R^{dN}\),
choose \(t\) so that:
\[
q_i+tp_i\in H.
\]
Then:
\[
q+tp-tR_H^{(i)}p=\sigma_H^{(i)}q.
\]
Hence:
\[
\begin{aligned}
f(q,p)
&=f(q+tp,p)
&&\text{(free-flight equation)}\\
&=f(q+tp,R_H^{(i)}p)
&&\text{(wall reflection)}\\
&=f(q+tp-tR_H^{(i)}p,R_H^{(i)}p)
&&\text{(free-flight equation)}\\
&=f(\sigma_H^{(i)}q,R_H^{(i)}p),
\end{aligned}
\]
and the proof follows. \qed
\end{proof}

\begin{lemma}[Pair symmetry]\label{lem:exchange}
Suppose that \(ij\) is an edge of \(\Gamma_\gamma\). Let
\(\tau_{ij}\) exchange the \(i\)-th and \(j\)-th balls. Then:
\[
f(\tau_{ij}q,\tau_{ij}p)=f(q,p),
\qquad q,p\in\R^{dN}.
\]
\end{lemma}
\begin{proof}
Set:
\[
\mathcal C_{ij}
=
\left\{
q\in\R^{dN}:\|q_i-q_j\|=2\rho
\right\}.
\]
Since \(ij\) is an edge of \(\Gamma_\gamma\),
\[
\mathcal C_{ij}\cap\overline{\Omega}_\gamma
\]
contains a nonempty relatively open subset of \(\mathcal C_{ij}\). For \(q\in\mathcal C_{ij}\), set:
\[
n=\frac{q_i-q_j}{2\rho}.
\]
On this subset the collision law gives:
\[
\begin{aligned}
p_i'&=p_i-\langle p_i-p_j,n\rangle n,\\
p_j'&=p_j+\langle p_i-p_j,n\rangle n,
\end{aligned}
\]
with all other momenta unchanged; hence:
\[
f(q,p')=f(q,p)
\]
on a nonempty relatively open subset of
\(\mathcal C_{ij}\times\R^{dN}\). Since both sides are polynomial and
\(\mathcal C_{ij}\) is connected, it 
follows that:
\[
f(q,p')=f(q,p)\quad \text{on the whole $\mathcal C_{ij}$.}
\]
Introduce the coordinates:
\[
x=\frac{q_i-q_j}{\sqrt2},
\qquad
v=\frac{p_i-p_j}{\sqrt2},
\]
and complete them to orthogonal 
coordinates:
\[
(x,y)\in\R^d\times\R^{dN-d},
\qquad
(v,w)\in\R^d\times\R^{dN-d}.
\]
Then the collision surface is \(\|x\|=a\), \(a = \sqrt{2} \rho\), and the collision law becomes:
\[
v\longmapsto
r_xv
=
v-2\frac{\langle v,x\rangle}{a^2}x.
\]
Thus:
\begin{equation}\label{eq:ball-reflection}
f(x,y,v,w)=f(x,y,r_xv,w),
\qquad \|x\|=a.
\end{equation}
The global free-flight equation 
allows us to consider the 
following formal billiard. The variable \(x\) moves freely inside the 
ball \(\|x\| \le a\), 
and is reflected at the sphere 
\(\|x\|=a\). The remaining position 
variables move freely:
\[
y(t)=y+tw,
\]
while \(w\) remains fixed. This need not be an actual trajectory of the
hard-ball system: other balls may overlap or leave the box. Nevertheless, \(f\) is constant along it. 
\smallskip%

Consider an initial state \((x,v)\) 
with \(\|x\|< a, v\neq 0\). Suppose that \(\|x\|<a\) and \(x,v\) are not 
collinear; then the \(x\)-motion stays in the two-dimensional plane spanned by \(x\) and
\(v\), and is an ordinary 
disc billiard. For a dense set of states \((x,v)\) the billiard trajectory 
is periodic. For some 
\(T = T(x,v) > 0\),
the trajectory returns to \((x,v)\) 
after time \(T\). Consequently,
\[
 f(x,y,v,w)
=
f(x,y+Tw,v,w).
\]
For fixed \(x,v,w\), the function \(y\mapsto f(x,y,v,w)\) is polynomial. If \(w\ne0\), the polynomial \(y\to f(x,y,v,w)\) has a nonzero
period; hence:
\begin{equation}\label{eq:period}
f(x,y+sw,v,w)=f(x,y,v,w) \quad s\in\R.
\end{equation}
Since periodic states are dense and \(f\) is polynomial, this identity
holds for all \(x,v,y,w\).
\smallskip%

Take a state \((x,v)\) such that the 
trajectory hits \(-x, -v\) after 
some time \(T\). Then:
\[
f(x,y,v,w) = f(-x,y + T w,-v,w).
\]
From \eqref{eq:period}: 
\[
f(x,y,v,w) = f(-x,y,-v,w).
\]
Such states are dense, so the identity holds everywhere. Finally,
\[
(x,y,v,w)\longmapsto(-x,y,-v,w)
\]
is precisely the simultaneous exchange of balls \(i\) and \(j\). This completes the proof. \qed
\end{proof}

\begin{lemma}[Wall propagation]\label{lem:wall-propagation}
Suppose that \(ij\) is an edge of \(\Gamma_\gamma\). If
\[
f(\sigma_H^{(i)}q,R_H^{(i)}p)=f(q,p),
\]
then:
\[
f(\sigma_H^{(j)}q,R_H^{(j)}p)=f(q,p).
\]
\end{lemma}
\begin{proof}
Take any \((q,p)\). First 
exchange the \(i\)-th and \(j\)-th balls:
\[
(q', p')
=
(\tau_{ij}q,\tau_{ij}p).
\]
By Lemma~\ref{lem:exchange},
\[
f(q', p')=f(q,p).
\]
Now reflect the \(i\)-th ball at the wall \(H\):
\[
f(\sigma_H^{(i)} q',
  R_H^{(i)} p')
=
f(q', p').
\]
Finally, exchange the \(i\)-th and \(j\)-th balls back. By
Lemma~\ref{lem:exchange}, this does not change \(f\). This completes the proof. \qed
\end{proof}

\begin{lemma}[Global symmetries]\label{thm:global}
Suppose that the collision graph \(\Gamma_\gamma\) is connected; then:
\[
f(\pi q,\pi p)=f(q,p)
\]
for every permutation \(\pi\) of the \(N\) balls. Moreover, for each wall \(H\) of the box and each \(i=1,\ldots,N\),
\[
f(\sigma_H^{(i)}q,R_H^{(i)}p)=f(q,p),
\qquad q,p\in\R^{dN}.
\]
\end{lemma}
\begin{proof}
By Lemma~\ref{lem:exchange}, \(f\) is invariant under the transposition
\(\tau_{ij}\) for every edge \(ij\) of \(\Gamma_\gamma\). Since
\(\Gamma_\gamma\) is connected, these transpositions generate all
permutations of the vertices. Hence
\[
f(\pi q,\pi p)=f(q,p)
\]
for every permutation \(\pi\). Fix a wall \(H\). Translate 
all balls toward \(H\) until one 
ball reaches the wall. By Lemma~\ref{lem:ball-wall}, \(f\) has
the wall symmetry for this ball. By Lemma~\ref{lem:wall-propagation}, this symmetry propagates along the
edges of \(\Gamma_\gamma\). Since the graph is connected, it holds for
all balls. \qed
\end{proof}

\begin{proof}[Proof of 
Theorem \ref{thm:gas}] 
We now show that \(f\) is independent of the position. Fix a ball \(i\)
and a coordinate \(k\), and let
\[
H^-=\{x_k=\rho\},
\qquad
H^+=\{x_k=L_k-\rho\}.
\]
By Lemma~\ref{thm:global}, \(f\) is invariant under the corresponding
reflections of the \(i\)-th ball. Their composition leaves the momenta
unchanged and translates only the coordinate \(q_{i,k}\):
\[
q_{i,k} \to q_{i,k}+2(L_k-2\rho).
\]
Hence:
\[
f\bigl(q+2(L_k-2\rho)e_{i,k},p\bigr)=f(q,p),
\]
where \(e_{i,k}\) denotes the corresponding coordinate vector in
\(\R^{dN}\). Then
\[
q_{i,k} \to f(q,p)
\]
is a polynomial with a nonzero period; therefore, it is constant. Since \(i\) and \(k\) were arbitrary, 
\[
f(q,p)=Q(p)\quad 
\text{for some polynomial \(Q\).}
\]
The rest follows from Kozlov's argument for the thin-gas case; see Theorem~11 in \cite{Kozlov}. At this stage, that argument uses only that \(f\) is independent of the position and invariant under permutations of the balls. Both properties have now been proved, so Kozlov's argument applies without change and shows that \(Q\) is a polynomial in the total kinetic energy. The difficulty in the dense case is precisely to establish these two properties. \qed

\end{proof}
\bibliographystyle{plain}
\bibliography{ref}

\end{document}